\documentclass[a4paper,centertags,oneside,12pt]{amsart}
\usepackage{mathrsfs}
\usepackage{appendix}
\usepackage{amssymb}
\usepackage{fancyhdr}
\usepackage{charter}
\usepackage{typearea}
\usepackage{pdfsync}
\usepackage{color}
\usepackage[colorlinks,linkcolor=blue,anchorcolor=blue,citecolor=green]{hyperref}
\usepackage{mathrsfs}
\usepackage[a4paper,top=3cm,bottom=3cm,left=3cm,right=3cm]{geometry}

\usepackage{enumitem}
\usepackage{color}
\setlist[1]{itemsep=5pt}
\newcommand{\comment}[1]{}

\makeatletter
      \def\@setcopyright{}
      \def\serieslogo@{}
      \makeatother

\newtheorem{theorem}{Theorem}[section]
\newtheorem{lemma}[theorem]{Lemma}
\newtheorem{corollary}[theorem]{Corollary}

\newtheorem{proposition}[theorem]{Proposition}

\newtheorem{definition}[theorem]{Definition}

\newtheorem{remark}[theorem]{Remark}

\numberwithin{equation}{section}

\begin{document}
\title{On the Bergman metric of a pseudoconvex domain with a strongly pseudoconvex polyhedral boundary point}
\author{Xiaojun Huang}
\address{Department of Mathematics, Rutgers University, New Brunswick, NJ 08903, USA.}
\email{huangx@math.rutgers.edu}
\thanks{Xiaojun Huang is supported in part by NSF DMS-2247151
}
\author{Scott James}
\address{Department of Mathematics, Rutgers University, New Brunswick, NJ 08903, USA.}
\email{sjj70@scarletmail.rutgers.edu}
\author{Xiaoshan Li}
\address{School of Mathematics and Statistics, Wuhan University, Wuhan, Hubei 430072, China.}
\email{xiaoshanli@whu.edu.cn}
\thanks{Xiaoshan Li is supported in part by NSFC (12361131577, 12271411)}

\begin{abstract}

 \centerline{\it Dedicated to  Professors Harold  Boas and Emil  Straube on the occasion of their retirement} 

 \bigskip\bigskip
    Let $D\subset\mathbb{C}^n$ with $n>1$  be a pseudoconvex domain, possibly unbounded, that contains a non-smooth strongly pseudoconvex polyhedral boundary point. We show that the Bergman metric of $D$ is not Einstein.
\end{abstract}
\maketitle

\section{Introduction}

For any bounded domain $D\subset \mathbb{C}^n$, its Bergman metric is an  invariant Kähler metric.
Cheng and Yau \cite{CY80} proved  that every bounded pseudoconvex domain in $\mathbb{C}^n$ with a $C^2$-smooth boundary admits a unique complete K\"ahler–Einstein metric (up to a scaling factor)  which is also biholomorphically invariant.
Later,  Mok and Yau \cite{MY80}  removed the  boundary regularity assumption and proved the existence of such a metric for arbitrary bounded pseudoconvex domains.

A natural problem arising from these works is to determine under what circumstances these two important  invariant metrics  coincide. A classical conjecture  of Yau \cite{Yau} states that the Bergman metric of a bounded pseudoconvex domain is Einstein if and only if it is biholomorphic to a bounded homogeneous domain.  Earlier,
Cheng \cite{C79}  had conjectured that the Bergman metric of a smoothly bounded strongly pseudoconvex domain is K\"ahler–Einstein if and only if the domain is biholomorphic to the unit ball.
Cheng's conjecture was confirmed in dimension two by Fu–Wong \cite{FW97} and Nemirovski--Shafikov \cite{NS06}, and was resolved in all dimensions by Huang–Xiao \cite{HX16} based on earlier work of many authors.
Subsequent generalizations were obtained for Stein manifolds and Stein spaces with compact strongly pseudoconvex boundaries; see Huang–Li \cite{HL23}, Ebenfelt–Xiao–Xu \cite{EXX22,EXX24}, and references therein.
Related variations of Cheng’s conjecture were  also discussed by S. Li in his papers \cite{L1, L2, L3}.

In a more recent development, Savale and Xiao \cite{SX23} investigated Bergman –Einstein metrics on smoothly bounded  pseudoconvex domains in $\mathbb{C}^2$.
They proved that a smoothly bounded pseudoconvex domain of finite type in  $\mathbb{C}^2$, whose Bergman metric is Einstein, must be biholomorphic to the unit ball in $\mathbb{C}^2$.
A prior result by Fu–Wong \cite{FW97} established an analogous statement for smoothly bounded complete Reinhardt pseudoconvex domains of finite type in  $\mathbb{C}^2$.

Despite these advances, not much is known about the Einstein property of Bergman metrics on unbounded pseudoconvex domains or in bounded pseudoconvex domains with  rough boundary points. In this paper, we aim to conduct a study along these lines.
We will show that the Bergman metric of a pseudoconvex domain, possibly unbounded, which possesses a strongly pseudoconvex polyhedral boundary point (as defined in Definition \ref{def1-9-24}) is not Einstein.
\begin{theorem}\label{main theorem}
Let $\Omega\subset\mathbb{C}^{n}$, with $n>1$,  be a (possibly unbounded) pseudoconvex domain.  
If $\Omega$  possesses a  strongly pseudoconvex polyhedral boundary point,
 \ then the Bergman metric of $\Omega$ is well-defined in a nonempty open subset of $\Omega$, denoted by $\Omega^* $,  and the Bergman metric  cannot be  Einstein on any open subset of  $\Omega^*$.
\end{theorem}
\begin{corollary}
Let $\Omega\subset\mathbb{C}^{n}$, with $n>1$,  be a bounded pseudoconvex domain.  If $\Omega$  possesses a  strongly pseudoconvex polyhedral boundary point,
then its  Bergman metric cannot be Einstein.
\end{corollary}

One of the main tools used in the proof of Theorem \ref{main theorem} is the rescaling argument, which has been  used  to work on many related 
problems. In particular, in connection with our present work, we 
mention   the papers by Wong \cite{W77},
Kim \cite{Kim},  Kim-Yu \cite{KY}, Krantz-Yu \cite{KYu}
and Boas-Straube-Yu \cite{BSY},  where the rescaling method has been used to study the boundary limit of  various quantities associated with  the Bergman metric. Indeed, our current work  has benefited   from  their studies.  A recent application of the rescaling method can also be found in Huang-Zhu \cite{HZ}, where it is employed in solving a CR transversality problem. Another recent  application of the rescaling method was used in working on the pinched  properties of a K\"ahler metric is a recent  paper of Bracci-Gauthier-Zimmer \cite{BGZ}.

The  ideas of our proof of the main theorem can be stated briefly as follows: First, we   show that if the Bergman metric of our domain is Einstein, then its Bergman invariant function is constant and equals that of the unit ball. Next we carefully construct a special sequence of points approaching a strongly pseudoconvex polyhedral boundary point  tangentially such that the limit domain is equivalent to the product of a ball and a bidisk of lower dimension.  
To obtain such a sequence, we assign weight $2$ to one of the complex normal directions, weight $1.5$  to the other normal directions, and weight $1$  to the remaining CR directions. The main part of the paper is then devoted to showing that the Bergman invariant function of this product domain coincides with that of  $\Omega$.
A direct computation shows that the Bergman invariant function of the unit ball differs from that of the aforementioned  product domain, leading to a contradiction.  In this respect, our proof departs from earlier approaches to the Cheng conjecture and its generalizations, where the contradiction is derived via spherical CR geometry and the Qi-Keng Lu uniformization theorem \cite{HX20}.

\section{Preliminaries}
Let $\Omega$ be a domain in $\mathbb{C}^{n}$. 
In what follows, we always assume that $n>1$.  
Write  $A^{2}(\Omega)$  for its Bergman space consisting of holomorphic functions on $\Omega$  that are square-integrable with respect to the Lebesgue measure.  We assume that $A^{2}(\Omega)\not =\{0\}$.  Then it  is  a non-trivial  Hilbert space.  
Let $\{\varphi_{j}\}_{j=1}^N$ be an orthonormal basis for  $A^{2}(\Omega)$   with respect to the standard inner-product,  where $N$   can be  finite or infinite.
The Bergman kernel  function  of   $\Omega$ is then defined by:
\[
K_{\Omega}(z,\overline{z})=\sum_{j=1}^{N}\varphi_{j}(z)\overline{\varphi_{j}(z)},\quad\forall z\in\Omega.
\]
The Bergman metric $g_{\Omega}$,   when well-defined,  is given by
\[
g_{\Omega}=\sum_{i,j=1}^{n}g_{i\overline{j}}\,dz_{i}\otimes d\overline{z}_{j},\quad\text{where}\quad g_{i\overline{j}}=\frac{\partial^{2}\log K_{\Omega}}{\partial z_{i}\partial\overline{z}_{j}};
\]
and the  Bergman  norm  is defined as
\[
g_{\Omega}(z,u)=\biggl(\sum_{i,j=1}^{n}g_{i\overline{j}}(z)\,u_{i}\overline{u}_{j}\biggr)^{1/2},\quad\forall u\in\mathbb{C}^{n}.
\]
The Bergman metric on a domain $\Omega$ in $\mathbb{C}^n$  may not be defined everywhere if $\Omega$ is not bounded.  When $\Omega$  is pseudoconvex   with a strongly pseudoconvex polyhedral boundary point, the Bergman metric is   indeed defined on  an open subset  of $\Omega$  near such a boundary point.    (See, e.g.,  \cite{GHH17} and \cite{James}).
 
 A necessary and sufficient condition for the existence of the Bergman metric at $z_0\in \Omega $ is that the Bergman space $A^{2}(\Omega)$ is base point free and separates holomorphic directions at $z_0$ \cite{Kob}.
 Here we recall that $A^{2}(\Omega)$ is said to be base-point free at $z_0$ if there is an element $f\in A^{2}(\Omega)$   such that $f(z_0)\not = 0$.    $A^{2}(\Omega)$ is said to separate holomorphic directions  at $z_0$ if for any holomorphic vector $X\in T_{z_0}^{(1,0)}{\Omega}$ there is an element $f\in A^{2}(\Omega)$   such that $X(f)(z_0)\not = 0$.

 In what follows,  we will write $\Omega^*$ for the subset
 of $\Omega$ consisting of points where  $A^{2}(\Omega)$ is  base-point free and separate holomorphic directions.  We will assume that  $\Omega^*\not = \emptyset.$

The Bergman canonical invariant function  is a positive real analytic function defined over $\Omega^*$ by
\[
J_{\Omega}(z):=\frac{\det G_{\Omega}(z)}{K_{\Omega}(z,\overline{z})},\quad\text{where}\quad G_{\Omega}(z)=\bigl(g_{i\overline{j}}(z)\bigr).
\]
The Ricci curvature tensor of the Bergman metric $g_{\Omega}$  is given by
\[
R_{\Omega}=\sum_{\alpha,\beta=1}^{n}R_{\alpha\overline{\beta}}\,dz_{\alpha}\otimes d\overline{z}_{\beta}\quad\text{with}\quad R_{\alpha\overline{\beta}}=-\frac{\partial^{2}\log\det G_{\Omega}}{\partial z_{\alpha}\partial\overline{z}_{\beta}},
\]
and the Ricci curvature along the  direction $u\in\mathbb{C}^{n}\setminus\{0\}$ is given by
\[
R_{\Omega}(z,u)=\frac{\sum_{\alpha,\beta=1}^{n}R_{\alpha\overline{\beta}}\,u_{\alpha}\overline{u}_{\beta}}{g_{\Omega}(z,u)^{2}}.
\]

The Bergman metric $g_{\Omega}$ is a Kähler metric over $\Omega^*$, and  is said to be Einstein in $\Omega^*$  if there exists a constant $c$ such that
\[
R_{\Omega}=c\,g_{\Omega}.
\]


We next recall the definition of    strongly pseudoconvex polyhedral boundary points for  a domain  $\Omega\subset{\mathbb C}^n$.
\begin{definition}\label{def1-9-24}
Let $\Omega$ be a possibly unbounded domain in $\mathbb{C}^{n}$  with $n>1$ and let $p\in\partial\Omega$.  
We say that $p$  is  a  strongly pseudoconvex polyhedral boundary  point if there exists a neighborhood $U$ of $p$ in $\mathbb C^n$ and $C^{2}$-smooth strongly plurisubharmonic functions $\rho_{1},\dots,\rho_{m}\colon U\to\mathbb{R}$ with $m>1$  such that
$\Omega\cap U=\{z\in U:\rho_{1}(z)<0,\dots,\rho_{m}(z)<0\}$  and 
$\{\partial\rho_{1}|_p,\dots,\partial\rho_{m}|_p\}$  are linearly independent over $\mathbb{C}$.
\end{definition}


\section{Localization of extremal domain functions on a possibly unbounded domain}
Let $\Omega$ be a domain in $\mathbb C^n$ with  $z\in\Omega^*$. Here we recall that $\Omega^*$ is the open subset of $\Omega$ where $A^2(\Omega)$ is base point free and separates holomorphic directions. We always assume that $\Omega^*\not =\emptyset$.  Then $K_{\Omega}(z,\overline z)>0$ and  the metric matrix $G_{\Omega}(z)$ is invertible for $z\in \Omega^*$.
We have  the following:
\begin{equation}\label{e1-3-30}
\begin{split}
    &K_\Omega(z,\overline z)=\sup\{|f(z)|^2:f\in A^2(\Omega),\|f\|=1\},\\[2mm]
    &g^2_{\Omega}(z,u)=\frac{1}{K_{\Omega}(z,\overline z)}\sup\biggl\{\Bigl|\sum_{j=1}^n u_j\frac{\partial f(z)}{\partial z_j}\Bigr|^2:\|f\|_{\Omega}=1,\,f(z)=0\biggr\}.
\end{split}
\end{equation}
Further define the following  extremal domain functions (see, e.g, \cite{KYu} \cite{James}):
\begin{equation}\label{e1-3-30}
\begin{split}
\\[2mm]
    &\lambda_{\Omega}^k(z):=\sup\Bigl\{\Bigl|\frac{\partial f}{\partial z_k}(z)\Bigr|:\|f\|_{\Omega}=1,\,f(z)=0,\,\frac{\partial f}{\partial z_j}(z)=0\ (1\le j<k)\Bigr\},\\[2mm]
    &I_{\Omega}(z,u):=\sup\Bigl\{u f''(z)[\overline G_{\Omega}]^{-1}(z)\overline{f''(z)}u^*:\|f\|_{\Omega}=1,\,f(z)=f'(z)=0\Bigr\}.\\[2mm]
\end{split}
\end{equation}

Here $\|\cdot\|_{\Omega}$ denotes the $L^2$-norm on $\Omega$, $u^*$ is the conjugate-transpose of the vector $u$, $f''(z)=\bigl(\frac{\partial^{2}f}{\partial z_i\partial\overline z_j}(z)\bigr)$.

The following formulas  relates  the above quantities to the Bergman canonical invariant function  $J_{\Omega}$ and  the Ricci curvature of the Bergman metric.

\begin{proposition}
    \cite{KYu}\label{Prop 1-3-31}
Let $\Omega$ be a domain in $\mathbb C^n$ with $K_{\Omega}(z,\overline z)>0$ and $G_{\Omega}(z)$ invertible at $z\in\Omega^*$.  Then
\begin{enumerate}
\item $J_{\Omega}(z)=\dfrac{\lambda_{\Omega}(z)}{K^{n+1}_{\Omega}(z,\overline z)}$ where $\lambda_{\Omega}(z)=\lambda^1_{\Omega}(z)\cdots\lambda^n_{\Omega}(z)$;
\item $R_{\Omega}(z,u)=(n+1)-\dfrac{1}{g^2_{\Omega}(z,u)K_{\Omega}(z,\overline z)}I_{\Omega}(z,u)$ for  $u\in\mathbb C^n\setminus\{0\}$.
\end{enumerate}
\end{proposition}

Krantz-Yu \cite{KYu} proved the Proposition \ref{Prop 1-3-31} for bounded domains in  $\mathbb C^n$.   The proof for unbounded domains is the same (see, e.g.,  \cite{James}).
We next need to localize the above quantites over  a possibly  unbounded domain. To this aim,  we need the following version of   the H\"ormander $L^2$-estimates  \cite[Theorem 1.14]{H65}. (See also Theorem 5 of \cite{GHH17} or Theorem 2.3.3 of \cite{Hu}.)
\begin{proposition}\label{prop1}
Let $\Omega\subset\mathbb C^n$ be a possibly unbounded pseudoconvex domain, and let $\varphi:\Omega\to[-\infty,\infty)$ be a plurisubharmonic function.  Assume that
\begin{enumerate}
\item $U\subset\Omega$ is open and $\varphi(z)-c|z|^2$ is plurisubharmonic on $U$ for some constant $c>0$, and
\item $v\in L^2_{(0,1)}(\Omega,\varphi)$ is a $C^\infty$  $(0,1)$-form satisfying $\overline\partial v=0$ and $\mathrm{supp}\,v\subset U$.
\end{enumerate}
Then there exists a $C^\infty$ function $u$ on $\Omega$ such that $\overline\partial u=v$ and
\[
\int_{\Omega}|u|^2e^{-\varphi}\,dv\le\frac{1}{c}\int_{\Omega}|v|^2e^{-\varphi}\,dv.
\]
\end{proposition}


Recall that $p\in\partial\Omega$ is called a \emph{local peak point} if there exists a neighborhood $U$ of $p$  in $\mathbb{C}^n$ and a continuous function $f_p$ on $\overline{\Omega\cap U}$, holomorphic on $\Omega\cap U$, such that $f_p(p)=1$ and $|f_p(z)|<1$ for every $z\in\overline{\Omega\cap U}\setminus\{p\}$.

\begin{lemma}\label{Le1}
Let $\Omega$ be a  possibly unbounded pseudoconvex domain and let $p\in\partial\Omega$ be a local peak point.  Suppose that for a certain  connected open subset  $U\Subset  {\mathbb C}^n$ with  $p\in U$,  there is a bounded from the above plurisubharmonic function $\varphi:\Omega\to[-\infty,0)$ satisfying
\begin{equation}\label{property 1}
\varphi(z)>-c_0~\text{on}~U\cap\Omega,~\text{and on which}~ \varphi-c|z|^2 ~\text{is strongly plurisubharmonic}
\end{equation}
with some constants $c_0>0$ and $c>0$.  Then, after shrinking $U$ if necessary, one has the localization property: 
\begin{enumerate}
    \item $\lim_{z\to p}\frac{K_{\Omega}(z,\overline z)}{K_{\Omega\cap U}(z,\overline z)}=1$,
    \item $\lim_{z\to p}\frac{g_{\Omega}(z,u)}{g_{\Omega\cap U}(z,u)}=1\quad(\forall u\in\mathbb C^n\setminus\{0\})$, 
    \item $\lim_{z\to p}\frac{\lambda_{\Omega}(z)}{\lambda_{\Omega\cap U}(z)}=1.$
\end{enumerate}
\end{lemma}


\begin{remark}\label{remark 1}
 Nikolov  in  \cite[Lemma~1]{Ni02} constructed   a plurisubharmonic function \(\varphi\)  on $\Omega$ with the property as in ~(\ref{property 1}).  See also Lemma~\ref{le3-28-1} later in this paper when $p$ is a strongly pseudoconvex polyhedral boundary point .    Nikolov   established in \cite[Lemma~1]{Ni02}  the localization formulas (1) and (2).  With the help of  Proposition \ref{prop1}, (3)  can  be obtained following    more or less the same  argument as in \cite{KYu}.   For the reader’s convenience, we provide a  detailed proof  of this localization which follows. We also mention that a consequence of Lemma \ref{Le1} is that when $U$ is sufficiently small, $U\cap \Omega\subset \Omega^*$, namely, the Bergman metric of  $\Omega$ is well defined in  $U\cap \Omega$.   (This is  proved directly with the H\"ormander $L^2$-estimates in  Proposition \ref{prop1} in \cite{James}.)
\end{remark}

\begin{proof}
    Let $h$ be a local peak function for $p$. After shrinking $U$ we may assume that
\begin{equation}\label{10-20-a1}
h(p)=1,~|h(z)|_{\overline{U\cap{{\Omega}}}\setminus \{p\}}<1.
\end{equation}
Pick small neighborhoods $V_2\Subset V_1\Subset U$ of $p$. Then there exists a constant $0<b<1$ such that $|h|\leq b$ on $\overline{(U\setminus V_2)\cap\Omega}$.  Fix a cut-off function $\chi\in C_0^\infty (U)$ with $\chi\equiv 1$ on $V_1$ and $0\leq\chi\leq 1$ on $U$. 
    
Given any function $f\in A^2(U\cap \Omega)$, for each integer $l\geq 1$, set $\alpha=\overline\partial(\chi f h^l)$. Then $\alpha$ is a smooth $\overline\partial$-closed $(0, 1)$-form on $\Omega$ with $${\rm supp}~\alpha\subset(U\setminus V_1)\cap\Omega.$$ Choose $\chi_1\in C_0^\infty(V_1)$ that satisfies $\chi_1\equiv 1$ on $\overline V_2$ and $0\leq \chi_1\leq 1$ on $V_1$.  Assume that $V_1$ is sufficiently small such that $$|z-\xi|<1, \forall z\in V_1, \xi\in V_2.$$ 

Choose an integer \(m>0\) sufficiently large so that for $\xi\in V_2$ the function
\[
\Phi(z)=m\varphi(z)+(2n+4)\chi_1(z)\log|z-\xi|
\]
is plurisubharmonic and non-positive on \(\Omega\), and in addition \(\Phi(z)-c|z|^2\) is plurisubharmonic on \(U\cap\Omega\).
Applying Proposition~\ref{prop1}, we obtain a smooth solution~$g$ of the equation $\overline\partial g=\alpha$ on~$\Omega$ satisfying

\begin{equation}\label{9-19-a1}
\int_{\Omega}|g|^{2}\,e^{-\Phi}\,dv
\le\frac{1}{c}\int_{\Omega}|\alpha|^{2}\,e^{-\Phi}\,dv.
\end{equation}
Since $\varphi>-c_{0}$ on~$U\cap\Omega$, the right-hand side is dominated by

\[
\frac{c_{1}}{c}\int_{\Omega\cap(U\setminus V_{1})}|h|^{l}|f|^{2}\,dv
\le\frac{c_{1}}{c}\,b^{l}\|f\|_{\Omega\cap U}
\]
for some constant $c_{1}>0$.  
Moreover, \eqref{9-19-a1} forces the following vanishing property: 

\begin{equation}\label{e1-3-27}
g(\xi)=0
\quad\text{and}\quad
\frac{\partial g}{\partial z_{j}}(\xi)=0
\quad\text{for }1\le j\le n.
\end{equation}
On the other hand, $\Phi<0$ on $\Omega$, so
\[
\int_{\Omega}|g|^{2}\,dv
\le\int_{\Omega}|g|^{2}\,e^{-\Phi}\,dv
\le\frac{c_{1}}{c}\,b^{l}\|f\|_{\Omega\cap U}
=:c_{2}\,b^{l}\|f\|_{\Omega\cap U}.
\]
Now set $F_\ell=\chi f h^\ell-g$ for $\ell\geq 1$. Then $F_\ell\in A^2(\Omega)$ and $F_\ell(\xi)=h^\ell(\xi)f(\xi)$. Moreover,
$$\|F_\ell\|_{\Omega}\leq (1+c_2b^\ell)\|f\|_{\Omega\cap U}.$$


For  $1\leq k\leq n$, let $f$ be an extremal function for $\lambda_{\Omega\cap U}^k(\xi)$ with $\xi\in V_2$.  Thus
    $$f\in A^2(\Omega\cap U),~\|f\|_{\Omega\cap U}=1,~f(\xi)=0,~\frac{\partial f}{\partial z_j}(\xi)=0, ~1\leq j<k$$
    and $$\lambda_{\Omega\cap U}^k(\xi)=\left|\frac{\partial f}{\partial z_k}(\xi)\right|.$$
  By (\ref{e1-3-27}), the function $F_\ell$ at $\xi$ satisfies
  \begin{equation}
      F_\ell(\xi)=0,~\frac{\partial F_\ell}{\partial z_j}(\xi)=0, 1\leq j<k.
  \end{equation}
    Consequently, 
    \begin{equation}
        \lambda_{\Omega}^k(\xi)\geq \frac{|\frac{\partial F_\ell}{\partial z_k}(\xi)|}{\|F_\ell\|}=\frac{|\frac{\partial f}{\partial z_k}(\xi)||h(\xi)|^\ell}{\|F_\ell\|}\geq \lambda^k_{\Omega\cap U}(\xi)\frac{|h(\xi)|^{\ell}}{(1+c_2 b^\ell)}.
    \end{equation}
    First let $\xi\rightarrow p$ and then $\ell\rightarrow\infty$; since $|h(p)|=1$, we obtain $$\liminf _{\xi\rightarrow p}\frac{\lambda_{\Omega}^k(\xi)}{\lambda^k_{\Omega\cap U}(\xi)}\geq 1.$$
   On the other hand, the monotonicity of $\lambda^k_{\Omega}$ gives $\lambda_{\Omega}^k(\xi)\leq \lambda_{\Omega\cap U}^k(\xi)$, hence $$\limsup_{\xi\rightarrow p}\frac{\lambda_{\Omega}^k(\xi)}{\lambda_{\Omega\cap U}^k(\xi)}\leq 1.$$
Therefore the third localization property holds:
\[
\lim_{\xi\to p}\frac{\lambda^{k}_{\Omega}(\xi)}{\lambda^{k}_{\Omega\cap U}(\xi)}=1.
\]
\end{proof} 
By virtue of the localization of the Bergman kernel $K_{\Omega}(z,\overline z)$ and the quantity $\lambda_{\Omega}(z)$, the localization of the Bergman canonical invariant $J_{\Omega}$ follows immediately from Proposition~\ref{Prop 1-3-31}.
\begin{corollary}\label{coro1-3-31}
    Under the same assumptions as in Lemma \ref{Le1},
    $$\lim_{z\rightarrow p}\frac{J_{\Omega}(z)}{J_{\Omega\cap { U}}(z)}=1.$$
    \end{corollary}
With Proposition \ref{prop1} in hand, the localization of $I_{\Omega}$ for an unbounded domain then follows  from the same argument of \cite[Proposition 2.4]{KYu}. We state  the result in the following lemma, omitting the details of the proof which can be found in \cite{James}.
    \begin{lemma}\label{10-20-lem1}
           Under the same assumptions as in Lemma \ref{Le1}, $$\lim_{z\rightarrow p}\frac{I_{\Omega}(z, u)}{I_{\Omega\cap U}(z, u)}=1, \forall u\in\mathbb C^n\setminus\{0\}.$$
    \end{lemma}
When $p$ is a $C^{2}$-smooth strongly pseudoconvex boundary point, the localization of $J_{\Omega}$ and $I_{\Omega}$, together with the results in \cite[Corollary 2]{KYu}, yields the following.

\begin{corollary}\label{coro1-3-30}
Let $\Omega$ be a pseudoconvex domain (possibly unbounded) in $\mathbb C^n$ and let $p\in\partial\Omega$ be a $C^{2}$ strongly pseudoconvex point. Then
\[
\lim_{z\to p}J_{\Omega}(z)=\frac{(n+1)^n\pi^n}{n!}
\quad\text{and}\quad
\lim_{z\to p}R_{\Omega}(z,u)=-1
\quad\text{for every }u\in\mathbb C^n\setminus\{0\}.
\]
\end{corollary}
\begin{proof}
Let $p\in\partial\Omega$ be a strongly pseudoconvex boundary point. Then it is a local peak point. By Remark \ref{remark 1}, there exists a bounded from above plurisubharmonic function $\varphi: \Omega\rightarrow(-\infty, 0)$ and a suitable neighborhood $U$ of $p$ such that the assumptions in  Lemma \ref{Le1} are satisfied.  Since $U\cap\Omega$ is a bounded pseudocovnex domain, it then follows from \cite[Corollary 2]{KYu} that $$\lim_{z\rightarrow p}J_{\Omega\cap U}(z)=\frac{(n+1)^n\pi^n}{n!},\quad\text{and}\quad  \lim_{z\rightarrow p}R_{\Omega\cap U}(z, u)=-1, \forall u\in\mathbb C^n\setminus\{0\}.$$ The first conclusion of the corollary then follows directly from Corollary \ref{coro1-3-31}. For the second part, fix $u\in\mathbb C^n\setminus\{0\}$. 
Since \begin{equation*}
    \begin{split}
        -R_{\Omega}(z, u)+(n+1)&=\frac{I_{\Omega(z,~ u)}}{g^2_{\Omega}(z,u)K_{\Omega}(z,\overline z)}\\
        &=\frac{I_{\Omega(z,~ u)}}{g^2_{\Omega}(z,u)K_{\Omega}(z,\overline z)}\cdot\frac{g^2_{\Omega\cap U}(z, u)K_{\Omega\cap U}(z, z)}{I_{\Omega\cap U}(z, u)}\cdot \frac{I_{\Omega\cap U(z,~ u)}}{g^2_{\Omega\cap U}(z,u)K_{\Omega\cap U}(z,\overline z)}\\
        &=\frac{I_{\Omega(z,~ u)}}{g^2_{\Omega}(z,u)K_{\Omega}(z,\overline z)}\cdot\frac{g^2_{\Omega\cap U}(z, u)K_{\Omega\cap U}(z, z)}{I_{\Omega\cap U}(z, u)}\cdot[-R_{\Omega\cap U}(z, u)+(n+1)]
    \end{split}
\end{equation*}
By Proposition \ref{Prop 1-3-31} (2), Lemma \ref{Le1} and Lemma \ref{10-20-lem1}, we have
\begin{equation}
\lim_{z\rightarrow p}[-R_{\Omega}(z, u)+(n+1)]=n+2
\end{equation}
and thus, $\lim_{z\rightarrow p}R_{\Omega}(z, u)=-1.$
This completes the proof.
\end{proof}

\begin{proposition}\label{prop1-4-3}
Let $\Omega$ be a possibly unbounded pseudoconvex domain in $\mathbb C^n$  which is stronlgy pseudoconvex polyhedral at some boundary point $p\in\partial\Omega$. Let $U$ be a neighborhood of $p$ such that $U\cap\Omega$ is connected, and on which the Bergman metric $g_{\Omega}$ is well defined. Then its Bergman metric $g_\Omega$ is K\"ahler-Einstein on $U\cap\Omega$  if and only if its Bergman canonical invariant $J_{\Omega}\equiv (n+1)^n\frac{\pi^n}{n!}$ on $U\cap\Omega$.    
\end{proposition}
\begin{proof}
By Corollary \ref{coro1-3-30}, for any smooth boundary point $q\in U\cap\partial\Omega$ near which $\partial\Omega$ is strongly pseudoconvex, one has 
    \begin{equation}\label{e2-3-31}
    \lim_{z\rightarrow q}J_{\Omega}(z)=\frac{(n+1)^n\pi^n}{n!}.
    \end{equation}
    The Bergman metric  $g_{\Omega}$ {is}  K\"ahler-Einstein if its Ricci curvature $R_{\Omega}=c g_{\Omega}$ for some constant $c$. By Corollary \ref{coro1-3-30}, one has $c=-1$. Consequently, the K\"ahler-Einstein assumption implies that $\log J_{\Omega}$ is a pluriharmonic function on $U\cap\Omega$. Now, for any attached holomorphic disk $\phi:\Delta\rightarrow\Omega$ where $\phi$ is holomorphic in $\Delta:=\{t\in {\mathbb C}: |t|<1\}$,  continuous up to $\overline {\Delta}$, and $\phi(\partial\Delta)$ is contained in the smooth part of $ U\cap \partial\Omega$, we have that $\log
      J_{\Omega}(\phi(t))$ is harmonic. Since it is constant on the strongly pseudoconvex part of the boundary by \eqref{e2-3-31}, it  assumes the value
\[
\log\frac{(n+1)^n\pi^n}{n!}
\]
everywhere on $\Delta$. Now, since $\partial \Omega$ is strongly pseudoconvex near $q$, the union of such disks fills up an open subset of $\partial\Omega$ near $q$. Since $\log
      J_{\Omega}$ is well defined in $U\cap\Omega$  on which it is real analytic, we conclude that $\log J_{\Omega}\equiv
      \log\frac{(n+1)^n\pi^n}{n!}$ over $U\cap\Omega$ as  $U\cap\Omega$ is connected by definition. 
Conversely, if $J_\Omega(z)$ takes a constant value near $p$, then the Bergman metric is obviously K\"ahler-Einstien. Thus, we have the conclusion of the proposition.
\end{proof}
\begin{remark}\label{remark1-10-21}

Note that the zero set of the Bergman kernel function, denoted by $E$, is a complex analytic variety in $\Omega$. 
Thus,  $J_\Omega$ is a well-defined real-analytic function on  $\Omega\setminus E$. Since $\Omega\setminus E$  is connected,  $J_\Omega$ is constant if and only if it is constant on some nonempty open subset of $\Omega$. In particular, when $\Omega$ contains a  $C^2$-smooth strongly pseudoconvex boundary point, the Bergman metric of the  domain  $\Omega$  is Kähler–Einstein  wherever  it is well-defined  if and only if  $J_\Omega=c$ is a constant on a certain open subset of   $\Omega\setminus E$ .
 In this case, $c=\frac{(n+1)^n\pi^n}{n!}$, and the Bergman space $A^2(\Omega)$ separates holomorphic directions at any point in $\Omega\setminus E$ and  thus the Bergman metric is well-defined in $\Omega\setminus E$.
\end{remark}
\section{Stability of the Bergman kernels}\label{sec2}

Our proof of Theorem \ref{main theorem} depends in part on the interior stability of the Bergman kernel functions first proven by Ramadanov \cite{Ra67}. (See also \cite{Kim}, \cite{Boas}, \cite{James}). The classical argument in \cite{Ra67} also proves the following.
\begin{proposition}[Ramadanov]\label{prop1-4-9-2025}
Let $D$ be a bounded  domain in $\mathbb C^n$
containing the origin. Let $\{D_s\}_{s=1}^\infty$ be a sequence of bounded domains in
$\mathbb C^n$ {whose closures converge to the closure of the bounded domain $D$ in
the sense of Hausdorff set convergence in such a way} {such} that for any $\varepsilon > 0$
there exists $N>0$ such that for any $s>N$ we have
$$(1-\varepsilon)D\subset D_s\subset(1+\varepsilon)D.$$
Then the sequence of the Bergman kernels $\{K_{D_s}\}$ converges uniformly to $K_{D}$ in the $C^\infty$-topology on any compact subset of $D$.
\end{proposition}
We next present a  normalization  of $\Omega$ near a  strongly pseudoconvex polyhedral boundary point, which will be crucial in our  rescaling argument:
\begin{lemma}\label{9-23-lem1}
Let $\Omega\subset\mathbb{C}^{n}$ be a domain with $p\in\partial\Omega$ being a strongly pseudoconvex polyhedral boundary point.
Then there exists a coordinate chart $(V,z)$ centered at $p$ and smooth functions $\{\Phi_{j}\}_{j=1}^{m}\in C^{\infty}(V)$ such that
\[
\Omega\cap V=\bigl\{z\in V:\operatorname{Im}z_{j}>\Phi_{j}(z,\overline z),\;1\le j\le m\bigr\}
\]
with
\[
\Phi_{j}(z,\overline z)=\sum_{\alpha,\beta=1}^{n}a^{j}_{\alpha\overline\beta}\,z_{\alpha}\overline z_{\beta}+R_{j}(z,\overline z),\qquad j=1,\dots,m,
\]
where $(a^{j}_{\alpha\overline\beta})$ are positive definite constant matrices for $1\leq j\leq m$. In particular, when $j=1$, $$\sum_{\alpha, \beta=1}^na^1_{\alpha\overline\beta}z_{\alpha}\overline{z_\beta}=|z''|^2+|z_1|^2+|z'P_{m-1}|^2$$ where $z''=(z_{m+1}, \cdots, z_n), ~z'=(z_2, \cdots, z_m)$ and $P_{m-1}$ is a constant invertible matrix of order $m-1$. Moreover  $R_{j}(z,\overline z)=\mathcal{O}(|z|^{3})$ for every $1\le j\le m$.
\end{lemma}
\begin{proof}
    After a change of coordinates, we assume that $p=0$ and 
\[
\Omega\cap V=\{z\in V:\rho_{1}(\xi)<0,\dots ,\rho_{m}(\xi)<0\},\qquad m\geq 2,
\]
with $\rho_{1}(0)=\dots =\rho_{m}(0)=0$ and $\{\partial\rho_{1}(0),\dots,\ \partial\rho_{m}(0)\}$ being  linearly independent over~$\mathbb C$.  Here each $\rho_j$ is strongly plurisubharmoinc  near  $p=0$. With a further   change of  holomorphic  coordinates and a re-choice of  defining functions~$\{\rho_{j}\}$ if necessary, we assume that  near~$p$
\[
\rho_{j}=-\operatorname{Im}\xi_{j}+H_{j}(\xi_{1},\dots ,\xi_{n}),\qquad 1\leq j\leq m,
\]
where
\begin{enumerate}
   \item $H_1=\sum_{\alpha\neq 1, \beta\neq 1}a_{1, \alpha\overline\beta}\xi_{\alpha}\overline{\xi_\beta}+{\mathcal O}(|\xi|^3)$,
   \item $(a_{1, \alpha\overline{\beta}})$  is a positive definite Hermitian $(n-1)\times(n-1)$-matrix,
   \item $H_j=\sum_{\alpha, \beta=1}^n a_{j, \alpha\overline\beta}\xi_{\alpha}\overline{\xi_\beta}+\mathcal O(|\xi|^3)$ and each $(a_{j, \alpha\overline\beta})$ is a positive definite Hermitian $n\times n$ matrix for $2\leq j\leq n$.
   \end{enumerate}
In what follows, we denote by $\mathcal{O}(3)$ any term satisfying $\mathcal{O}(3) = \mathcal{O}(|\xi|^3)$ as $\xi \to 0$.
There exists an invertible matrix $C_{n-1}$ of order $n-1$ such that $$C_{n-1}(a_{1, \alpha\overline\beta})_{(n-1)\times(n-1)}\overline {C_{n-1}^t}=I_{n-1}$$ where $C_{n-1}^t$ is the transpose of $C_{n-1}$ and $I_{n-1}$ is the identity matrix of order $n-1$. Choose a new coordinates $\tilde z=(\tilde z_1, \tilde z_2, \cdots, \tilde z_n)$ with 
\begin{equation*}
    \begin{cases}
        \xi_1=\tilde z_1,\\
        (\xi_2, \cdots, \xi_n)=(\tilde z_2, \cdots, \tilde z_n)C_{n-1}.
    \end{cases}
\end{equation*}
With respect to the new coordinates $\tilde z$, we have
\begin{equation*}
    \begin{cases}
        \rho_1=-{\rm Im ~\tilde z_1}+\sum_{j=2}^n|\tilde z_j|^2+\mathcal O(3),\\
        \rho_j=-{\rm Im}~l_j(\tilde z_2, \cdots, \tilde z_n)+\sum_{\alpha, \beta=1}^n\tilde a^j_{\alpha\overline\beta} \tilde z_{\alpha}\overline{\tilde z_{\beta}}+\mathcal O(3).
    \end{cases}
\end{equation*}
Here, $$l_j(\tilde z_2, \cdots, \tilde z_n)=(\tilde z_2, \cdots, \tilde z_n)\alpha_j^t$$
where $\alpha_2^t, \cdots, \alpha_m^t$ are $\mathbb C$-linearly independent constant vectors of length $n-1$. Choose orthonormal vectors
$\beta_{m+1}^{t}, \beta_{m+2}^t, \beta_{n}^t$ such that \begin{equation}\label{11-11-a1}
\beta_j^t\perp\alpha_k^t, \quad m+1\leq j\leq n, ~2\leq k\leq m.
\end{equation} Then we choose $\beta_2^t, \cdots, \beta_m^t$ which extend $\{\beta_{m+1}^t, \cdots, \beta_{n}^t\}$ to a unitary matrix $$U_{n-1}=(\beta_2^t, \cdots, \beta_m^t, \beta_{m+1}^t, \cdots, \beta_n^t).$$
Now we define new coordinates $\hat z$ with 
\begin{equation*}
    \begin{cases}
        \tilde z_1=\hat z_1,\\
        (\tilde z_2, \cdots, \tilde z_n)=(\hat z_2, \cdots, \hat z_n)\overline{U^t_{n-1}}.
    \end{cases}
\end{equation*}
By the orthogonal property (\ref{11-11-a1}), then 
with respect to $\hat z$, we have $$l_j(\hat z)=(\hat z_2, \cdots, \hat z_m)\hat \alpha_j^t, \quad 2\leq j\leq m$$
where $\{\hat\alpha_j^t\}_{j=2}^m$ with each $\hat\alpha_j^t$ the first $m-1$ components of the vector $\overline{U^t_{n-1}}\alpha_j^t$ are $\mathbb C$-linearly independent vectors. Choose an invertible matrix $P_{m-1}$ such that $$P_{m-1}(\hat\alpha_2^t, \cdots, \hat\alpha_m^t)=I_{m-1}.$$
Then we choose new coordinates $w$ with 
\begin{equation*}
    \begin{cases}
        \hat z_1=w_1\\
        (\hat z_2, \cdots, \hat z_m)=(w_2, \cdots, w_m)P_{m-1}\\
        (\hat z_{m+1}, \cdots, \hat z_n)=(w_{m+1}, \cdots, w_n).
    \end{cases}
\end{equation*}
With respect to the new coordinates $w$, we have
\begin{equation*}
    \begin{cases}
        \rho_1=-{\rm Im}~w_1+\sum_{j=m+1}^n|w_j|^2+|(w_2, \cdots, w_m)P_{m-1}|^2+\mathcal O(3),\\
        \rho_j=-{\rm Im}~w_j+\sum_{\alpha, \beta=1}^nb^j_{ \alpha\overline\beta}w_\alpha\overline{w_\beta}+\mathcal O(3), \quad 2\leq j\leq m,
    \end{cases}
\end{equation*}
where $(b^j_{\alpha\overline\beta})_{n\times n}$ are positive Hermitian matrices for $2\leq j\leq m$.
Since $$w_1\overline{w_1}=w_1(w_1-2iv_1)=w_1^2-2iw_1v_1,$$ then on $\{\rho_1=0\}\cap V$ we have
$$|w_1|^2={\rm Im} (iw_1^2)+{\mathcal O}(|w|^3)$$
which implies that
$${\rm Im}(w_1+iw_1^2)=\sum_{j=m+1}^n|w_j|^2+|w_1|^2+|(w_2, \cdots, w_m)P_{m-1}|^2+\mathcal O(3).$$
After a further coordinates change $z_1=w_1+i w^2, z_2=w_2, \cdots, z_n=w_n$, we get
\[
\Phi_1(z,\overline z)=\sum_{\alpha={m+1}}^{n}|z_\alpha|^2+|z_1|^2+|(z_2, \cdots, z_m)P_{m-1}|^2+{\mathcal O}(3).
\]
We complete the proof.
\end{proof}

Let $\Phi_j=\sum_{\alpha, \beta=1}^n a^j_{\alpha\overline\beta}z_{\alpha}\overline z_\beta+R_j,\quad 1\leq j\leq m$ be as in Lemma \ref{9-23-lem1}. 
Write 
\begin{equation}\label{11-23-a1}
U_0=\{z\in\mathbb C^n: |z_j|<\varepsilon_0, j=1,\dots,n\}\end{equation}  where  $\varepsilon_0\ll1$  is  such that on $U_0$ one has
$$ |R_j(z)|\leq \frac{A_0}{2}|z|^2, \quad 1 \leq j \leq m, \quad \forall z\in U_0,$$
where 
\begin{equation}\label{10-27-a3}
A_0:=\min\{A_j: 1\leq j\leq m\}
\end{equation}
with $A_j$ being  the minimum eigenvalue of the matrix $(a^j_{\alpha\overline\beta})$. Furthermore, we may assume that 
$$y_1>\frac12 \sum_{j=m+1}^n|z_j|^2, \quad \forall z\in U_0\cap\Omega.$$
We define an in-homogenous tangential  scaling map $L_{\delta}$ as follows:
For $z\in\mathbb C^n$, $${L_{\delta}(z_1, \dots, z_m,z_{m+1},\dots,z_n)=(\delta^{-2}z_1, \delta^{-\frac32}z_2, \dots, \delta^{-\frac32}z_m, \delta^{-1}z_{m+1},\dots,\delta^{-1}z_n)}$$ with $0<\delta\ll1$. Write 
$$\widetilde z_1=\delta^{-2}z_1, \quad \widetilde z_j=\delta^{-\frac32}z_j, \quad \widetilde z_k=\delta^{-1}z_k, \quad 2 \leq j \leq m, ~~{m+1} \leq k \leq n.$$ 
and  $$D_0:=\Omega\cap U_0, \quad \widetilde D_{\delta, 0}:=L_{\delta}(D_0).$$ 
Then $$\widetilde D_{\delta, 0}\subset \widetilde D^\ast:=\{(\widetilde z_1, \widetilde z_2,\cdots, \widetilde z_n): \widetilde y_1>\frac 12 \sum_{j={m+1}}^n|\widetilde z_j|^2, ~\widetilde y_k>0,2\leq k \leq m\}$$ as when $\widetilde z\in \widetilde D_{\delta, 0}$,  we have
\[\begin{cases}
    &\widetilde y_1>\frac12(\delta^2|\widetilde z_1|^2+\delta c_1\sum_{j = {2}}^m|\widetilde z_j|^2+\sum_{j = {m+1}}^n|\widetilde z_j|^2)\geq\frac12\sum_{j = {m+1}}^n|\widetilde z_j|^2,\\
    &\widetilde y_k>\frac{A_0}{2}(\delta^{\frac52}|\widetilde z_1|^2+\delta^{\frac32}\sum_{j = {2}}^m|\widetilde z_j|^2+\delta^{\frac12}\sum_{j = {m+1}}^n|\widetilde z_j|^2)>0, 2 \leq k \leq m
\end{cases}\]
Here, $c_1$ is a constant  depending only on $P_{m-1}$ which is given in Lemma \ref{9-23-lem1}.
Define the  linear fractional  transformation $\Phi$ as follows: 
$$\Phi(\widetilde z_1, \dots, \widetilde z_n)=\left(\frac{\widetilde z_1-i}{\widetilde z_1+i},\frac{\widetilde z_2-i}{\widetilde z_2+i}, \dots, \frac{\widetilde z_m-i}{\widetilde z_m+i}, \frac{2\widetilde z_{m+1}}{\widetilde z_1+i},\dots, \frac{2\widetilde z_n}{\widetilde z_1+i}\right).$$ Its inverse is given by  $$\Phi^{-1}(w_1, \dots, w_n)=\left(\frac{i(1+w_1)}{1-w_1},\frac{i(1+w_2)}{1-w_2}, \dots,  \frac{i(1+w_m)}{1-w_m}, \frac{iw_{m+1}}{1-w_1},\dots,\frac{iw_n}{1-w_1}\right)$$
with $$\widetilde z_j=\frac{i(1+w_j)}{1-w_j}, ~1\leq j\leq m, ~\widetilde z_k=\frac{iw_k}{1-w_1}, {m+1} \leq k \leq n$$
and $$\widetilde y_j=\frac{1-|w_j|^2}{|1-w_j|^2},\quad 1 \leq j \leq m,$$
where we use the notations $\widetilde z_j=\widetilde x_j+i\widetilde y_j, ~1\leq j\leq n$.
Set 
\begin{equation}
  \label{omega_0}  
\widehat\Omega_{\delta, 0}:=\Phi(\widetilde{D}_{\delta, 0}).
\end{equation}
Then $\widehat\Omega_{\delta, 0}\subset\Phi(\widetilde D^\ast) $ with 

\begin{equation}
\begin{split}
    \Phi(\widetilde D^\ast)
    &=\left\{(w_1, \dots, w_n): \frac{1-|w_1|^2}{|1-w_1|^2}>\frac12\frac{\sum_{j={m+1}}^n|w_j|^2}{|1-w_1|^2},~\frac{1-|w_k|^2}{|1-w_k|^2}>0, 2 \leq k \leq m\right\}\\
   & =\left\{(w_1, \dots, w_n):|w_1|^2+\frac12 \sum_{j={m+1}}^n|w_j|^2<1, ~~|w_k|<1, ~~2 \leq k \leq m\right\}
    \end{split}
\end{equation}
which shows that $\widehat\Omega_{\delta, 0}$ is a bounded domain for $0<\delta\ll1$.  

The main technical result of this section is the following:
\begin{proposition}\label{prop11}
Assume $m\geq 2$.
For any $\hat\varepsilon>0$, there exists a $\delta_0>0$ such that when $\delta<\delta_0$ one has
\begin{equation}\label{9-23-a2}
(1-\hat\varepsilon)(\mathcal{I}(\mathbb B^{n-m+1}\times \Delta^{m-1}))\subset\widehat\Omega_{\delta, 0}\subset (1+\hat\varepsilon) (\mathcal{I}(\mathbb B^{n-m+1}\times \Delta^{m-1}))
\end{equation}
where  $\mathbb B^{n-m+1}$ is the unit ball in $\mathbb C^{n-m+1}$ and $\Delta^{m-1}$ is the unit polydisk in $\mathbb C^{m-1}$ and 
$\mathcal I(w_1, w_{m+1}, \cdots, w_n, w_{2}, \cdots, w_m)=(w_1, \cdots, w_n)$. Hence,
$$\mathcal {I}(\mathbb B^{n-m+1}\times\Delta^{m-1})=\{w\in\mathbb C^n: |w_1|^2+\sum_{j=m+1}^n|w_j|^2<1,\quad |w_k|<1, 2\leq k\leq m\}.$$
\end{proposition}
\begin{proof}
First, for compact subsets $K_1\Subset \mathbb B^{n-m+1}, K_2\Subset \Delta^{m-1}$ we will verify that $$\mathcal I(K_1\times K_2)\subset \widehat\Omega_{\delta, 0}$$ when $\delta$ is sufficiently small.

Fix any small $c>0$,  there exists an $0<\varepsilon_c\ll\varepsilon_0$ where $\varepsilon_0$ is given in (\ref{11-23-a1}), such that
$$|R_j(z)|\leq c|z|^2, \quad 1\leq j\leq m,   \text{when }\quad |z_k|<\varepsilon_c, ~1\leq k\leq n.$$
Denote $$D_{\varepsilon_c}:=\Omega\cap\{(z_1, \cdots, z_n)\in\mathbb C^n: |z_k|<\varepsilon_c, \quad k=1, \cdots, n\}.$$Hence, $\widetilde D_{\varepsilon_c, \delta}:=L_{\delta}(D_{\varepsilon_c})$ contains the following
\begin{equation*}
\begin{cases}
    \widetilde y_1>(1+c)(\delta^2|\widetilde z_1|^2+\delta c_2\sum_{j = {2}}^m|\widetilde z_j|^2+\sum_{j = {m+1}}^n|\widetilde z_j|^2)\\
    \widetilde y_k>(A^*_0+c)(\delta^{\frac52}|\widetilde z_1|^2+\delta^{\frac32}\sum_{j = {2}}^m|\widetilde z_j|^2+\delta^{\frac12}\sum_{j = {m+1}}^n|\widetilde z_j|^2), ~2\leq k\leq m.\\
    |\widetilde z_1|<\delta^{-2}\varepsilon_c; \quad |\widetilde z_j|<\delta^{-\frac32}\varepsilon_c, ~2\leq j\leq m; \quad|\widetilde z_k|<\delta^{-1}\varepsilon_c, ~m+1\leq k\leq n.
\end{cases}
\end{equation*}
Here, $c_2>0$ depends only on $P_{m-1}$  given in Lemma \ref{9-23-lem1} and
$$A^*_0:=\max\{A^*_j: 2\leq j\leq m\}$$ 
with $A^*_j$ being  the maximum eigenvalue of the matrix $(a^j_{\alpha\overline\beta})$.
For $R>0$, write

\begin{equation*}
\widetilde D_{\varepsilon_c, R}^\ast:=
    \begin{cases}
        \widetilde y_1>(1+c)\sum_{j = {m+1}}^n|\widetilde z_j|^2+\eta_{1, R}\delta,\\
        \widetilde y_k>\eta_{2, R}\delta^{\frac12},~2\leq k\leq m,\\
        |\widetilde z_j|<R, 1\leq j\leq n.
        
    \end{cases}
\end{equation*}
where 
\begin{equation*}
    \begin{cases}
        \eta_{1, R}=(1+c)[\delta R^2+(m-1)c_2R^2],\\
        \eta_{2, R}=(A^*_0+c)[\delta^2+(m-1)\delta+(n-m)]R^2.
    \end{cases}
\end{equation*}
When $R<\delta^{-1}\varepsilon_c$,  it holds that
$$\widetilde D_{\varepsilon_c, R}^\ast\subset \widetilde D_{\varepsilon_c, \delta}.$$
Then
\begin{equation*}
    \Phi(D_{\varepsilon_c, R}^\ast)=
    \begin{cases}
        \frac{1-|w_1|^2}{|1-w_1|^2}>(1+c)\frac{\sum_{j=m+1}^n|w_j|^2}{|1-w_1|^2}+\eta_{1, R}\delta,\\
        \frac{1-|w_k|^k}{|1-w_k|^2}>\eta_{2, R}\delta^{\frac12}, ~2\leq k\leq m,\\
        |\frac{1+w_j}{1-w_j}|<R, ~1\leq j\leq m,\\
        ~\frac{|w_\ell|}{|1-w_1|}<R,~m+1\leq \ell\leq n.
    \end{cases}
\end{equation*}
That is,
\begin{equation*}
    \Phi(D^\ast_{\varepsilon_c, R})=\begin{cases}
        |w_1|^2+(1+c)\sum_{j=m+1}^n|w_j|^2+\eta_{1, R}\delta(|1-w_1|^2)<1,\\
        |w_k|^2+\eta_{2, R}\delta^{\frac12}|1-w_k|^2<1, ~2\leq k\leq m,\\
        |1+w_j|<R|1-w_j|, 1\leq j\leq m,\\
        |w_\ell|<R|1-w_1|, ~m+1\leq \ell\leq n.
    \end{cases}
\end{equation*}
For any $0<\varepsilon'\ll 1$, let  $R$ be a fixed but sufficiently large number such  that
\[
R>\frac{2}{\varepsilon'}.
\]  
Then there exists a  $\delta_{0}>0$ such that for all $\delta<\delta_{0}$ one has
\begin{align*}
    &\eta_{1, R}\delta(1-|w_1|^2)<1-(1-\varepsilon')^2, \\
    &\eta_{2, R}\delta^{\frac12}|1-w_k|^2<1-(1-\varepsilon')^2\ \text {for} \  2\leq k\leq m.
   \end{align*}

It follows that $\Phi(D^\ast_{\varepsilon_c, R})$ contains the set
\begin{equation*}
    \begin{cases}
    |w_1|^2+(1+c)\sum_{j=m+1}^n|w_j|^2<(1-\varepsilon')^2,\\
    |w_k|^2<(1-\varepsilon')^2, \quad 2\leq k\leq m,
     \\
        |1+w_j|<R|1-w_j|, 1\leq j\leq m,\\
        |w_\ell|<R|1-w_1|, ~m+1\leq \ell\leq n.
\end{cases}
\end{equation*}
which, when  $\varepsilon'\ll 1$,  contains the compact set $\mathcal I(K_1\times K_2)$. Hence we first let $R$ be sufficiently large, then we can find a $\delta_0>0$ sufficiently small such that when $0<\delta<\delta_0$,   $\widehat\Omega_{\delta, 0}$ 
contains $\mathcal I(K_1\times K_2)$ whenever $\delta<\delta_0$. Thus, 
we   conclude  the proof of the  first inclusion  in  (\ref{9-23-a2}).

In the following, we prove the second inclusion of (\ref{9-23-a2}). Since
\begin{equation}
    \widetilde D_{\varepsilon_c, \delta}:=L_{\delta}(D_{\varepsilon_c})
    \subset
    \left\{\widetilde y_1>(1-c)\sum_{j=m+1}^n|\widetilde z_j|^2; \qquad \widetilde y_k>0, \quad 2\leq k\leq m   \right\},
\end{equation}
we have 
\begin{equation}\label{10-27-a1}
\widehat\Omega_{\varepsilon_c, \delta}:=\Phi(\widetilde D_{\varepsilon_c, \delta})\subset\left\{|w_1|^2+(1-c)\sum_{j=m+1}^n|w_j|^2<1, \quad |w_k|<1, ~2\leq k\leq m\right\}.
\end{equation}

Now, we assume that $z\in D_0\setminus D_{\varepsilon_c}$. Then $|z_{j_0}|\geq \varepsilon_c$ for some $j_0$. Hence, after scaling by $L_{\delta}$, we have at least one of the following $n$ inequalities  $$|\widetilde z_1|>\frac{\varepsilon_c}{\delta^2};\quad |\widetilde z_j|>\frac{\varepsilon_c}{\delta^{\frac32}},\quad  2\leq j\leq m; \quad|\widetilde z_l|>\frac{\varepsilon_c}{\delta},\quad m+1\leq l\leq n.$$ Then after the linear fractional transformation $\Phi$, at least one of the following three cases holds:
\begin{enumerate}
    \item $|1+w_1|>\frac{\varepsilon_c}{\delta^2}|1-w_1|$,
    \item There exists a $j$ with $m+1\leq j\leq n$ such that 
    $|w_j|>\frac{\varepsilon_c}{\delta}|1-w_1|$,
    \item There exists a $j$ with $2\leq j\leq m$ such that $|1+w_j|>\frac{\varepsilon_c}{\delta^{\frac32}}|1-w_j|$.
\end{enumerate}
In the first or second case, one has $$|w_1-1|\leq \frac{\delta}{\varepsilon _c}.$$
Since $|w_1|^2+\frac12\sum_{j=m+1}^n|w_j|^2<1$,  we have $$|w_j|^2\leq \frac{4\delta}{\varepsilon_c}, m+1\leq j\leq n.$$ Hence, in the first or second case, we have 
$$|w_1-1|+\sum_{j=m+1}^n|w_j|^2\le \frac{4(n-m) \delta+\delta}{\varepsilon_c} \ \text{for}\ 0<\delta\ll1.$$
In the third case, first we  have
\begin{equation}\label{25-9-23a1}
\widetilde y_1>\frac12\left(\delta^2|\widetilde z_1|^2+\delta c_1\sum_{j=2}^m|\widetilde z_j|^2+\sum_{j=m+1}^n|\widetilde z_j|^2\right); \quad\widetilde y_j>0,\quad 2\leq j\leq m.
\end{equation}
The third case is equivalent to $|\widetilde z_j|>\frac{\varepsilon_c}{\delta^{\frac32}}$ for some $ 2\leq j\leq m$. Then from (\ref{25-9-23a1}) we have $\widetilde y_1>\frac{\varepsilon^2_c}{2\delta^2}\rightarrow\infty$ as $\delta\rightarrow0$ which implies that $$|w_1-1|\leq \frac{2\delta}{\varepsilon_c}.$$
Combining with $|w_1|^2+\frac{1}{2}\sum_{j=m+1}^n|w_j|^2<1$,  we once more have   $$\sum_{j=m+1}^n|w_j|^2\leq c^\ast \frac{\delta}{\varepsilon_c} ,$$
where $c^\ast$ is a constant independent of $\varepsilon_c, \delta$, which may be different in each different context.
Thus, in the third case, we still have
$$|w_1-1|+\sum_{j=m+1}^n|w_j|^2\le c^\ast \frac{ \delta}{\varepsilon_c}, \ 0<\delta\ll1,$$
Hence, when $z\in D_0\setminus D_{\varepsilon_c}$, that is, for $w=\Phi\circ L_{\delta}(z)\in \widehat\Omega_{\delta, 0}\setminus \widehat\Omega_{\varepsilon_c, \delta}$ we  have
\begin{equation}\label{10-27-a2}
|w_j|<1, \quad 2\leq j\leq m;  \quad |w_1-1|+\sum_{j=m+1}^n|w_j|^2\le c^\ast \frac{{\delta}}{\varepsilon_c}, \ \text{for}\ 0<\delta\ll1.
\end{equation}
Then (\ref{10-27-a1}) and (\ref{10-27-a2}) imply the second inclusion of (\ref{9-23-a2}).  The proof of  Proposition \ref{prop1} is complete.
\end{proof}

An immediate consequence of Proposition \ref{prop1-4-9-2025} and Proposition \ref{prop1} is the following
\begin{corollary}\label{coro1-9-24}
The Bergman kernel $K_{\widehat\Omega_{\delta, 0}}$ converges to $K_{\mathcal{I}(\mathbb B^{n-m+1}\times \Delta^{m-1})}$ uniformly in the $C^\infty$-topology on compact subsets of $\mathcal{I}(\mathbb B^{n-m+1}\times \Delta^{m-1})$. As a consequence, $$J_{\widehat\Omega_{\delta, 0}}\rightarrow J_{\mathcal{I}(\mathbb B^{n-m+1}\times \Delta^{m-1})}, \quad \text{as}~\delta\rightarrow 0$$ on any compact subset of $\mathcal{I}(\mathbb B^{n-m+1}\times \Delta^{m-1})$. 
\end{corollary} 


\section{Bergman-Einstein metrics on unbounded pseudoconvex domains}
We now proceed to the proof of Theorem \ref{main theorem}.

\begin{proof}[Proof of Theorem \ref{main theorem}]
Let $\Omega$ be a pseudoconvex domain which is  strongly pseudoconvex polyhedral at a boundary point $p\in\partial\Omega$ as defined in Definition~\ref{def1-9-24}.  After shrinking $U$ if necessary, we may assume that
\begin{enumerate}
\item[(i)]  there are $C^{2}$-smooth strongly plurisubharmonic functions $\{\rho_{j}\}_{j=1}^{m}$ with $m>1$ on $U$ such that
\begin{equation}\label{9-24-a2}
    U\cap\Omega=\{z\in U:\rho_{1}(z)<0,\dots,\rho_{m}(z)<0\};
\end{equation}
\item[(ii)] the vectors $\partial\rho_{1}(q),\dots,\partial\rho_{m}(q)$ are linearly independent over $\mathbb C$ for $q\in U\cap \overline\Omega$.
\end{enumerate}
By Lemma \ref{25-9-23a1}, after a suitable change of coordinates on a small neighborhood $V\Subset U$ of $p$, we may assume that $p=0$ and 
\begin{equation}\label{9-24-a1}
V\cap\Omega=\{(z_1,\cdots, z_n)\in V: {\rm Im~ }z_1>\Phi_1(z, \overline z), \cdots, {\rm Im}~ z_m>\Phi_m(z, \overline z)\}
\end{equation}
where $$\Phi_1(z, \overline z)=|z|^2+R_1(z), \Phi_j(z, \overline z)=\sum a^j_{\alpha\overline\beta} z_{\alpha}\overline z_{\beta}+R_j(z), \quad 2\leq j\leq m$$
with each remainder $R_j=\mathcal O(|z|^3)$. Write $$U_0=\{z\in\mathbb C^n: |z_j|<\varepsilon_0, 1\leq j\leq n\}$$  with $\varepsilon_0\ll 1$  such that $U_0\Subset V$ and on $U_0$ one has
$$|R_j(z)|\leq \frac{A_0}{2}|z|^2, \forall z\in U_0, 1\leq j\leq m,$$
where $A_0$ is defined in (\ref{10-27-a3}).
We first construct a bounded continuous plurisubharmonic function $\psi$ in $\Omega$ where $\psi$ is strictly plurisubharmonic near $p=0$ as follows: 
\begin{lemma}\label{le3-28-1}
After shrinking $U_0$ if necessary, there exists a plurisubharmonic function  $\psi:\Omega\rightarrow (-\infty, 0)$  such that $$\psi (z)>-c_0, ~\left(\frac{\partial^2\psi(z)}{\partial z_j\partial\overline z_k}\right)\geq c ~I_{n}, ~~z\in U_0\cap\Omega$$for some constants $c_0>0,~c>0$.
\end{lemma}
\begin{proof}
Let $(V, z)$ be the coordinates given as in (\ref{9-24-a1}) and set $\varphi=\frac{A_0}4|z|^2-y_1$. Then $\varphi$ is strictly plurisubharmonic on $V$ and satisfies $\varphi(0)=0, \varphi(z)<0$ when $z\in U_0\cap\overline\Omega\setminus \{0\}$.  Take $r>0$ such that $\overline {\mathbb B^n(0, r)}\Subset U_0$. Set  $M=\max\{\varphi(z): z\in \partial \mathbb B^n(0, r)\cap\overline \Omega\}$.  Then $M<0$. Now we define $\psi$ as follows:
\begin{equation}
    \psi=\begin{cases}
        \max\{\varphi(z), M\}, &z\in \mathbb B^n(0, r)\cap\Omega\\
        M, &z\in\Omega\setminus \mathbb B^n(0, r).
    \end{cases}
\end{equation}
Then $\psi$ is a bounded and continuous plurisubharmonic function on $\Omega$ with $$\psi(0)=0, ~~\psi(z)<0, \forall z\in\overline\Omega\setminus\{0\}.$$
Furthermore, $\psi$ is equal to $\varphi$ near $0$ with $$\left(\frac{\partial^2\psi}{\partial z_i\partial\overline z_j}\right)= \frac{A_0}4 I_n $$
in some neighborhood $U$ of $0$ in $\mathbb C^n$. Moreover, $\psi>-c_0$ on $U$ for some positive constant $c_0$.
\end{proof}

It follows from Lemma \ref{Le1} and Corollary \ref{coro1-3-31} that 
\begin{corollary}\label{coro1-4-2}
    After shrinking the neighborhood $U_0$ given in Lemma \ref{le3-28-1}, we have the localization of the  Bergman canonical invariant: 
    $$\lim_{z\rightarrow p}\frac{J_{\Omega}(z)}{J_{\Omega\cap U_0}(z)}=1.$$
\end{corollary}
Assume that the Bergman metric on $\Omega^*$ is K\"ahler--Einstein.  
Proposition~\ref{prop1-4-3} then gives
\[
\lim_{z\to p}J_{\Omega\cap U_0}(z)=\frac{(n+1)^n\pi^n}{n!}.
\]  
 In the following, we show that, as $z\to p$, the limit of the Bergman canonical invariant $J_{\Omega\cap U_0}(z)$ depends on the local geometry of $\Omega$ near $p$ which  will produce a contradiction. 
 We next  prove the following:
 \begin{lemma}\label{le2-9-24}
\begin{equation}
    (n+1)^n\frac{\pi^n}{n!}\neq \frac{(n-m+2)^{n-m+1}2^{m-1}\pi^n}{(n-m+1)!}, \quad\text{for}~ n\geq m\geq 2.
\end{equation}
\end{lemma}
\begin{proof}
Let $k = n - m + 1$. Then $1 \leq k \leq n-1$, and the above becomes $$\pi^n \frac{(n+1)^n}{n!} \neq \pi^{k}\frac{(k+1)^k}{k!}(2\pi)^{n-k}$$ After obvious cancellation, we then need to prove
$$\frac{(n+1)^n}{2^n n!} \neq \frac{(k+1)^k}{2^k k!}$$ We will prove  this by showing that the sequence $a_n = \frac{(n+1)^n}{2^n n!}$ is strictly increasing. We calculate
\begin{align*}
\frac{a_{n}}{a_{n-1}} = \frac{(n+1)^{n}}{2^{n}n!} \frac{2^{n-1} (n-1)!}{n^{n-1}} = \frac{(n+1)^n}{2n^n}
\end{align*}
Recall the Arithmetic Mean--Geometric Mean inequality: $$\frac{x_1 + \cdots + x_n}{n} \geq \sqrt[\leftroot{-3}\uproot{3}n]{x_1 \cdots x_n}$$
and  the   equality holds  if and only if all the $x_j$ are the same. 

Let $x_1 = 2$, $x_2 = \cdots = x_n = 1$.
This gives $\frac{n+1}{n} > \sqrt[\leftroot{-3}\uproot{3}n]{2}$. Thus, $a_n / a_{n-1} > 1$ and  the sequence is strictly increasing.
\end{proof}
    Since we assumed that the Bergman metric on $\Omega$ is K\"ahler-Einstein on $\Omega^*$ for some neighborhood $U$ of $p$ in $\mathbb C^n$ and there are $C^2$-strongly pseudoconvex boundary points of $\partial\Omega$ near $p$, 
     by Proposition~ \ref{prop1-4-3},  we have  $$J_{\Omega}(z)\equiv \frac{(n+1)^n\pi^n}{n!}$$ on $U\cap\Omega$. It follows from the Corollary~\ref{coro1-4-2} that there exists a neighborhood $U_0\Subset U$ of $p$ in $\mathbb C^n$ such that $$\lim_{z\rightarrow p} J_{\Omega\cap U_0}(z)=\frac{(n+1)^n\pi^n}{n!}.$$ 
    Let the scaling map $L_{\delta}$ and the fractional linear mapping $\Phi$ be given as in section ~{\ref{sec2}}. Recall that $\widehat\Omega_{\delta, 0}=\Phi\circ L_{\delta}(\Omega\cap U_0)$. By Corollary ~\ref{coro1-9-24} one has
    \begin{equation}\label{9-24-a3}
    J_{\widehat\Omega_{\delta, 0}}\rightarrow J_{\mathcal{I}(\mathbb B^{n-m+1}\times \Delta^{m-1})}, \quad \text{as}~\delta\rightarrow 0\end{equation} 
    on any compact subset of $\mathcal I(\mathbb B^{n-m+1}\times \Delta^{m-1})$, where $m\geq 2$ is the positive integer number given in (\ref{9-24-a2}).
    We choose the coordinates $z$ defined in Lemma \ref{9-23-lem1}  and define $$\xi_{\delta}:=(i\delta^2, i\delta^{\frac32},\cdots, i\delta^{\frac32}, 0,\cdots, 0) .$$ Then $\xi_{\delta}\in\Omega\cap U_0$ when $\delta$ is sufficiently small and
$\xi_{\delta}\rightarrow 0,~\text{as}~\delta\rightarrow 0.$
Thus by the localization of the Bergman canonical invariant,
$$\lim_{\delta\rightarrow 0}J_{\Omega\cap U_0}(\xi_{\delta})=(n+1)^n\frac{\pi^n}{n!}.$$
On the other hand, since $J_{\Omega}$ is biholomorphically invariant and by Corollary \ref{coro1-9-24} we have
$$J_{\Omega\cap U_0}(\xi_{ \delta})=J_{\widehat\Omega_{\delta, 0}}(0, \cdots, 0)\rightarrow J_{\mathcal{I}(\mathbb B^{n-m+1}\times \Delta^{m-1})}( 0, \cdots, 0), \quad\text{as}~\delta\rightarrow 0.$$
This implies that 
\begin{equation}\label{9-24-a5}
J_{\mathcal{I}(\mathbb B^{n-m+1}\times \Delta^{m-1})}(0, \cdots, 0)=(n+1)^n\frac{\pi^n}{n!}.
\end{equation}
Since $K_{\mathbb B^{n-m+1}\times \Delta^{m-1}}=K_{\mathbb B^{n-m+1}}\cdot K_{\Delta^{m-1}}$, one has
$J_{\mathbb B^{n-m+1}\times \Delta^{m-1}}=J_{\mathbb B^{n-m+1}}\cdot J_{\Delta^{m-1}}.$ By a direct calculation, $$J_{\mathbb B^{n-m+1}}\equiv \frac{(n-m+2)^{n-m+1}\pi^{n-m+1}}{(n-m+1)!},\quad J_{\Delta^{m-1}}\equiv (2\pi)^{m-1}.$$ Thus,
\begin{equation}\label{9-24-a6}
    J_ {\mathbb B^{n-m+1}\times \Delta^{m-1}} \equiv \frac{(n-m+2)^{n-m+1}2^{m-1}\pi^n}{(n-m+1)!}.
\end{equation}
Since  $\mathcal I(\mathbb B^{n-m+1}\times\Delta^{m-1})$ is biholomorphic to $\mathbb B^{n-m+1}\times\Delta^{m-1}$ by the map $$\mathcal I(w_1, w_{m+1}, \cdots, w_n, w_{2}, \cdots, w_m)=(w_1,  \cdots, w_n),$$  we have $J_{\mathbb B^{n-m+1}\times\Delta^{m-1}}(0)=J_{\mathcal I(\mathbb B^{n-m+1}\times\Delta^{m-1})}(0).$
It follows from (\ref{9-24-a5}) and (\ref{9-24-a6}) that 
\begin{equation*}
    (n+1)^n\frac{\pi^n}{n!}=\frac{(n-m+2)^{n-m+1}2^{m-1}\pi^n}{(n-m+1)!}, n\geq m\geq 2.
\end{equation*}
This  contradicts Lemma \ref{le2-9-24}. By Remark \ref{remark1-10-21}, the Bergman metric can not be K\"ahler-Einstein in any open subset  of  $\Omega^\ast$. 
\end{proof}

\end{document}